\documentclass[a4paper,11pt]{article}
\usepackage{amsmath}
\usepackage{amsfonts}
\usepackage{amsthm}
\usepackage{amssymb}
\usepackage{graphicx}
\usepackage{pstricks}

\makeatletter

\newtheorem{theorem}{Theorem}

\newtheorem{prop}[theorem]{Proposition}
\newtheorem{obs}[theorem]{Observation}
\newtheorem{cor}[theorem]{Corollary}

\newtheorem{exm}[theorem]{Example}
\newtheorem{prm}[theorem]{Problem}
\newtheorem{rmk}[theorem]{Remark}
\newtheorem{conj}[theorem]{Conjecture}
\newtheorem*{pf}{Proof}

\def\thm{\textbftheorem}
\def \bp {\begin{prop} \ }
\def \ep {\end{prop}}

\def \bpm {\begin{prm} \ }
\def \epm {\end{prm}}

\def \bc {\begin{cor} \ }
\def \ec {\end{cor}}

\def \bo {\begin{obs} \ }
\def \eo {\end{obs}}

\def \thm {\begin{theorem} \ }
\def \ethm {\end{theorem}}

\def \bl {\begin{lem} \ }
\def \el {\end{lem}}

\def \bd {\begin{defi} \ \rm }
\def \ed {\end{defi}}

\def \brm {\begin{rmk} \ }
\def \erm {\end{rmk}}

\def \bxm {\begin{xmp} \ \rm }
\def \exm {\end{xmp}}

\def \bcj {\begin{conj}}
\def \ecj {\end{conj}}

\def \nmr {\begin{enumerate}}
\def \enmr {\end{enumerate}}

\def \tmz {\begin{itemize}}
\def \etmz {\end{itemize}}

\begin{document}

\title{Characterizations of $2$-Colorable (Bipartite) and $3$-Colorable Graphs}

\author{E. Sampathkumar and M. A. Sriraj\\
\normalsize Department of Mathematics, University of Mysore, Mysore\\
\normalsize India\\
\normalsize \vspace{1ex} Email: esampathkumar@eth.net;
srinivasa$_{-}$sriraj@yahoo.co.in\\ }

\maketitle

\begin{abstract}
\qquad A \emph{directional labeling} of an edge $\emph{uv}$ in a
graph $G=(V,E)$ by an ordered pair $ab$ is a labeling of the edge
$uv$ such that the label on $uv$ in the direction from $u$ to $v$
is $\ell(uv)=ab$, and  $\ell(vu)=ba$. New characterizations of
$2$-colorable (bipartite) and $3$-colorable graphs are obtained in
terms of directional labeling of edges of a graph by ordered pairs
$ab$ and $ba$. In addition we obtain characterizations of
2-colorable and 3-colorable graphs in terms of matrices called
directional adjacency matrices.\\

\hbox{uniform} adjacency matrix.\\

\noindent {\bf 2010 AMS Classification:} 05C78, 05C15, 05B20
\end{abstract}

\section{Introduction}

\qquad For any definition on graphs we refer the book [1].\\

An edge $uv$ in a graph $G=(V,E)$ is \emph{directionally labeled}
by an \hbox{ordered} pair $ab$ if the label $\ell(uv)$ on $uv$ is
$ab$ in the direction from $u$ to $v$, and $\ell(vu)=ba$. Let
$uu_i$, $1\leq i \leq k$, be the edges at a vertex $u$ in $G$.
If~$\ell(uu_i)=ab$ for $1\leq i \leq k$, we say that the labelings
of the edges at $u$ is \emph{directionally uniform}, and we denote
this fact by $\ell_e(u)=ab$.\\

\textbf{Definition 1.} A graph $G=(V,E)$ admits a directionally
uniform edge labeling by ordered pairs $ab$ and $ba$ if for any
two adjacent vertices $v_1$ and $v_2$ in $G$, $\ell_e(v_1)=ab$
$\Longleftrightarrow$ $\ell_e(v_2)=ba$.\newpage

$\phantom{000}$\vspace{1.5cm}
\begin{center}
\begin{tabular}{@{}c@{\hskip 4cm}c@{}}
\begin{tabular}{@{\hskip -2cm}c}
\psline(0,.5)(2.5,.5)%
\psline(-1,-1)(0,.5)(-1,2)%
\psline(3.5,-1)(2.5,.5)(3.5,2)%
\rput[c](-1,-1){$\bullet$}%
\rput[c](0,.5){$\bullet$}%
\rput[c](-1,2){$\bullet$}%
\rput[c](3.5,-1){$\bullet$}%
\rput[c](2.5,.5){$\bullet$}%
\rput[c](3.5,2){$\bullet$}%
\rput[c](1.3,.75){$ba$}%
\rput[tr]{-55}(.3,1.6){\parbox{2cm}{\flushright $ab$}}
\rput[tr]{55}(2.4,1.8){\parbox{2cm}{\flushright $ab$}}
\rput[tr]{45}(-1,.4){\parbox{2cm}{\flushright $ab$}}
\rput[tr]{-55}(3.7,.2){\parbox{2cm}{\flushright $ab$}}
\\[2cm]
\end{tabular} &\begin{tabular}{c} \psframe(3,3) \rput[tr]{90}(-.8,1.7){\parbox{2cm}{\flushright
$ba$}} \rput(1.5,3.2){$ab$}
\rput[tr]{90}(2.65,1.7){\parbox{2cm}{\flushright $ab$}}
\rput(1.5,-.2){$ba$} \rput(-2,-.7){Figure 1}
\end{tabular}
\end{tabular}\vspace{1pc}
\end{center}

For example, the graphs in Figure~1 admit such edge labelings.\\\\
One can easily verify that an odd cycle does not admit such an edge labeling.\\

A graph $G=(V,E)$ is \emph{bipartite} if $V$ can be partitioned
into two sets $V_1$ and $V_2$ such that every edge in
$G$ joins a vertex in $V_1$ and a vertex in $V_2$.\\

The following is a new characterization of bipartite
graphs.\vspace{.3pc}
\begin{theorem}
For a graph $G=(V,E)$ the following statements are equi\-valent.\\
(i) $G$ is bipartite.\\
(ii) $G$ admits a uniform directional labeling of its edges by
ordered pairs $ab$ and $ba$.\vspace{.3pc}\end{theorem}
\begin{proof}(i) $\Rightarrow$ (ii). Since $G$ is bipartite, there exists a partition $V=V_1 \cup V_2$ of $V$ such that every edge in $G$ joins a vertex of $V_1$ and a vertex of $V_2$. Directionally label all its edges going from $V_1$ to $V_2$ by the ordered pair $ab$. Then for any edge $uv$ going from $V_1$ to $V_2$, $\ell_e(u)=ab$ and $\ell_e(v)=ba$. This implies that $G$ admits a directional labeling of its edges by the ordered pairs $ab$ and $ba$.\\
(ii)$\Rightarrow$ (i). Suppose there exists such a labeling of the edges of $G$ by ordered pairs $ab$ and $ba$. Let $V_1=\{u\in V:\ell_e(u)=ab\}$, and $V_2=\{v\in V:\ell_e(v)=ba\}$. Since for an two adjacent vertices $v_1$ and $v_2$, $\ell_e(v_1)\neq \ell_e(v_2)$, it follows that $V_1$ as well as $V_2$ are independent, and every edge in $G$ joins a vertex of $V_1$ and a vertex of $V_2$. Hence, $G$ is bipartite.\end{proof}
\section{$(1,2)$-\emph{Uniform directional labeling of the edges of a graph by ordered pairs $ab$ and $ba$}}
$2$-\emph{uniform directional labeling of the edges incident at a vertex}.\\

\qquad Let $vv_i$, $1\leq i \leq k$ be the edges incident to a vertex $v$ in a graph $G=(V,E)$.\\

\enlargethispage*{1pc}

\qquad If for some edges $vv_i$, $1\leq i \leq k$, $\ell(vv_j)=ab$, and some other edges $vv_j$, $1\leq j \leq k$, $\ell(vv_j) = ba$, we say that the directional labeling of the edges at $v$ by the ordered pairs $ab$ and $ba$ is $2$-\emph{uniform}. We denote this fact by
$\ell_e(v)=\{ab,ba\}$.\\\\
\textbf{Definition 2}. A graph $G=(V,E)$ \emph{admits a}
(1,2)-\emph{uniform directional labeling} of its edges by ordered
pairs $ab$ and $ba$ if for any two adjacent vertices $v_1$ and
$v_2$ one of the following holds:

\noindent (i) $\ell_e(v_1) = ab$, $\ell_e(v_2)=ba$\\
(ii)$\ell_e(v_1) = ab$, $\ell_e(v_2)=\{ab,ba\}$\\
(iii)$\ell_e(v_1) = ba$, $\ell_e(v_2)=\{ab,ba\}$.\\\\\\\\\\\\

\begin{center}
\hskip -3cm\begin{tabular}{c}
\psline(0,0)(3,0)(1.5,3)(0,0)%
\rput[c](0,0){$\bullet$}%
\rput[c](3,0){$\bullet$}%
\rput[c](1.5,3){$\bullet$}%
\rput[tr]{58}(.2,2.2){\parbox{2cm}{\flushright $ab$}}
\rput[tr]{-65}(3.05,1.8){\parbox{2cm}{\flushright $ab$}}
\rput(1.5,-.3){$ab$}
\end{tabular}
\end{center}\vspace{6cm}

\begin{center}
\hskip -2.5cm\begin{tabular}{@{}c@{\hskip 6cm}c@{}}
\begin{tabular}{c}
\psline(0,0)(3,0)(3,5)(0,5)(0,0) \psline(0,0)(1.5,1)(3,0)
\psline(0,5)(1.5,4)(3,5) \psline(1.5,1)(1.5,4)
\rput[c](0,0){$\bullet$} \rput[c](3,0){$\bullet$}
\rput[c](3,5){$\bullet$} \rput[c](0,5){$\bullet$}
\rput[c](1.5,1){$\bullet$} \rput[c](1.5,4){$\bullet$}
\rput[c](0,5.4){$V_{1}$} \rput[c](3,5.4){$V_{4}$}
\rput[c](0,-.4){$V_{2}$} \rput[c](3,-.4){$V_{3}$}
\rput[c](1.5,-.3){$ba$}\rput[tr]{90}(2.7,2.5){\parbox{3cm}{\flushright $ab$}}%
\rput[c](1.5,5.3){$ab$}
\rput[tr]{90}(-.8,2.5){\parbox{3cm}{\flushright $ba$}}
\rput[tr]{90}(.7,2.5){\parbox{3cm}{\flushright $ab$}}%
\rput[tr]{32}(.6,1.35){\parbox{3cm}{\flushright $ba$}}
\rput[tr]{-32}(2.7,1.2){\parbox{3cm}{\flushright $ab$}}
\rput[c](1.5,.5){$V_{6}$}
\rput[tr]{-38}(1.5,4.95){\parbox{2cm}{\flushright $ab$}}
\rput[tr]{37}(1.85,5.2){\parbox{5cm}{\flushright $ba$}}
\rput[c](1.8,3.8){$V_{5}$}
\end{tabular} &\begin{tabular}{c}
\psline(0,0)(3,0)(3,5)(0,5)(0,0) \psline(0,0)(1.5,1)(3,0)
\psline(0,5)(1.5,4)(3,5) \psline(1.5,1)(1.5,4)
\rput[c](0,0){$\bullet$} \rput[c](3,0){$\bullet$}
\rput[c](3,5){$\bullet$} \rput[c](0,5){$\bullet$}
\rput[c](1.5,1){$\bullet$} \rput[c](1.5,4){$\bullet$}
\rput[c](0,5.4){$V_{1}$} \rput[c](3.2,5.4){$V_{4}$}
\rput[c](0,-.4){$V_{2}$} \rput[c](3,-.4){$V_{3}$}
\rput[c](1.5,-.3){$ba$}%
\rput[tr]{90}(2.7,2.5){\parbox{3cm}{\flushright $ab$}}%
\rput[c](1.5,5.3){$ab$}
\rput[tr]{90}(-.8,2.5){\parbox{3cm}{\flushright $ba$}}
\rput[tr]{90}(.7,2.5){\parbox{3cm}{\flushright $ba$}}%
\rput[tr]{32}(.6,1.35){\parbox{3cm}{\flushright $ab$}}
\rput[tr]{-32}(2.7,1.2){\parbox{3cm}{\flushright $ba$}}
\rput[c](1.5,.5){$V_{6}$}
\rput[tr]{-38}(1.5,4.95){\parbox{2cm}{\flushright $ab$}}
\rput[tr]{37}(1.85,5.2){\parbox{5cm}{\flushright $ab$}}
\rput[c](1.8,3.8){$V_{5}$} \rput[c](0,6.3){$ab$}
\psline[linearc=2](0,0)(-1,1)(-1,3)(0,15)(3,5)
\end{tabular}\\[2pc]
\multicolumn{2}{@{\hskip 3cm}c@{}}{Figure 2}
\end{tabular}
\end{center}

\section{$3$-Colorable Graphs}

\qquad A graph $G=(V,E)$ is $3$-\emph{colorable} if we can color
its vertices with $3$ colors such that no two adjacent vertices
receive the same color, and this is not possible with
$2$-colors.\newpage
The following is a characterization  of $3$-colorable graphs.\\
\begin{theorem} For a graph $G=(V,E)$ the following statements are equi\-valent.\\
(i) $G$ is $3$-colorable.\\
(ii)$G$ admits a (1,2)-uniform directional edge labeling by
ordered pairs $ab$ and $ba$.\end{theorem}\vspace{.2pc}
\begin{proof} $(i) \Rightarrow (ii)$. Suppose $G$ is $3$-colorable. Then we can partition $V= V_1 \cup V_2 \cup V_3$ into three independent sets, where say, vertices in $V_i$, $1\leq i\leq 3$ are colored $i$, $1\leq i \leq 3$. Without loss of generality , we can assume that every vertex in $V_3$ has degree at least 2.
For, suppose for a vertex $v$ in $V_3$, $deg v = 1$, and $vv_i$ is an edge, and $v_i$ is colored $i$, $1\leq i \leq 2$. Then we can recolor $v$ with 1 or 2 according as $v_i$ is colored 2 or 1. This implies that we can shift $v$ to $V_1$ or $V_2$.\\
First, label all the edges going out of $V_1$ directionally by $ab$. Then all the edges going out from $V_2$ to $V_1$ are directionally labeled by $ba$. Now, label all the unlabeled edges going out from $V_2$ to $V_3$ by $ba$. We find that for any $u\in V_1$, $v \in V_2$ and $w \in V_3$, $\ell_e(u) = ab$, $\ell_e(v)=ba$, $\ell_e(w) = \{ab,ba\}$. This implies (ii).\\
(ii)$\Rightarrow$ (i). Suppose (ii) holds. Let $V_1= \{u:
\ell_e(u)=ab\}$, $V_2=\{v:\ell_e(v)= ba\}$, $V_3= \{w:\ell_e(w)=
ab,ba\}$. Since for any two adjacent vertices $v_1$ and $v_2$,
$\ell_e(v_1) \neq \ell_e(v_2)$, we see that the sets $V_1$, $V_2$
and $V_3$ are independent, and hence $G$ is
$3$-colorable.\end{proof}

\section{Directional adjacency}

\qquad The directional labeling of an edge in a graph by an ordered pair $ab$ can also be considered as directional adjacency as follows.\\

Let $G=(V,E)$ be a graph and $uv$ be an edge in $G$. Sometimes it may be necessary to distinguish the adjacency between $u$ and $v$ in the directions from $u$ to $v$ and from $v$ to $u$. For example, we can consider the adjacency from $u$ to $v$ as positive, and that from $v$ to $u$ as negative. This is denoted by $da(uv)=1$, and $da(vu)=-1$.\\\\
Suppose \\
$\bullet$ $A$, $B$ are two persons,\\
$\phantom{00}$ (a) $A$ is talking to $B$.\\
$\phantom{00}$ (b) $A$ is a boss and $B$ is a subordinate of $A$.\\
$\bullet$ $A$, $B$ are nodes in an electrical network, and current is flowing from $A$ to $B$.\\
$\bullet$ $A$ is a transmitter and $B$ is a receiver.

In all the above cases, we can consider the adjacency from $A$ to $B$ as positive and that from $B$ to $A$ as negative.\\
\section{Uniform directional adjacency($u.d. adjacency$)}

\enlargethispage*{1pc}

\qquad As in the case of uniform directional labeling of edges at a vertex, we define $u.d. adjacency$ at a vertex as follows.\\

\qquad Let $uu_i$, $1\leq i\leq k$ be edges at a vertex $u$ in $G$. We say that directional adjacency is \emph{uniform} at $u$ if $da(uu_i)=1$ or $-1$, for $1 \leq i \leq k$.\\

If the directional adjacency is uniform at $u$, we denote this fact by $da(u)= 1$ or $da(u)= -1$, accordingly.\\

A graph $G$ is said to have $u.d.adjacency$ if it has such an adjacency at each of its vertices.\\\\
One can easily verify the following:\\
\begin{prop} A graph $G$ has $u.d.adjacency$ if, and only if, for
any two adjacent vertices $u$ and $v$ $da(u)= 1
\Longleftrightarrow da(v)= -1$. Note that not all graphs have
$u.d.adjacency$.\end{prop}
For example, consider $K_3$ in Figure 3.\\

\hskip 2.7cm\begin{minipage}{5cm}
\unitlength 1mm 
\linethickness{0.4pt}
\ifx\plotpoint\undefined\newsavebox{\plotpoint}\fi 
\begin{picture}(107.75,53.25)(0,0)
\put(27.25,47.5){\circle*{1.5}} \put(13.75,20.25){\circle*{1.5}}
\put(42.25,20.5){\circle*{1.5}} \put(27.5,47){\line(-1,-2){13.75}}
\multiput(27.25,47.5)(.0336757991,-.0616438356){438}{\line(0,-1){.0616438356}}
\put(42,20.5){\line(-1,0){28.5}}
\put(26,51){$v_1$}
\put(8,18){$v_2$} \put(45,18){$v_3$} \put(24,13){Figure 3}
\end{picture}\vspace{-.7cm}
\end{minipage}

Suppose it has $u.d.adjacency$. Then say $da(v_1)= 1$. This implies $da(v_2)= -1$, $da(v_3)=1$ and $da(v_1) = -1$, which is a contradiction.\\
\section{Directional Adjacency Matrix of a Graph}

Let $G=(V,E)$ be graph of order $p$, and
$V=\{v_1,v_2,\ldots,v_p$\}. A \emph{directional adjacency matrix
(d.a.matrix)} $M=[a_{ij}]$ of $G$ is a matrix of order $p$ whose
entries $a_{ij}$ are defined as follows: For any two vertices
$v_i$,$v_j$ in $G$.

\noindent If $v_iv_j$ is an edge then $a_{ij}=1$ or $-1$. If $v_iv_j$ is an edge, then the directional adjacency from $v_i$ to $v_j$ is said to be positive or negative according as $a_{ij}=1$ or $-1$. The directional adjacency between $v_i$ and $v_j$ is \emph{symmetric} if $a_{ij}=a_{ji}=1$ or $-1$. Further, the directional adjacency between $v_i$ and $v_j$ is \emph{antisymmetric } if $a_{ij}=1$ $\Longleftrightarrow$ $a_{ji}= -1$. If $v_i$, $v_j$ are not adjacent $a_{ij}=0$.\\
Let $v_iv_1,v_iv_2,\ldots,v_iv_k$ be the edges at a vertex $v_i$. Then the directional \hbox{adjacency} at a vertex $v_i$ is said to be uniform if $a_{ii_1}=a_{ii_2}=\ldots=a_{ii_k} = 1$ or $-1$. Note that the directional adjacency at $v_i$ is uniform if, and only if, all the non zero entries in the $i^{th}$ row of the d.a.matrix $M$ are either $1$ or $-1$.\\
For example, the d.a. matrix $M$ of the tree $T$ in Figure 4 is
the uniform d.a.matrix of $T$.\\

\unitlength 1mm 
\linethickness{0.4pt}
\ifx\plotpoint\undefined\newsavebox{\plotpoint}\fi 
\noindent\begin{picture}(89.25,83.75)(0,0)
\put(4,72.25){\circle*{1.5}}
\put(3.5,48.25){\circle*{1.5}}
\put(14,59.75){\circle*{1.5}}
\put(39.25,72.5){\circle*{1.5}}
\put(28.5,59.75){\circle*{1.5}}
\put(38.25,47.5){\circle*{1.5}}
\multiput(4,72.25)(.0337370242,-.0423875433){289}{\line(0,-1){.0423875433}}
\multiput(13.75,60)(-.0336538462,-.0368589744){312}{\line(0,-1){.0368589744}}
\put(3.25,48.5){\line(0,-1){.25}}
\put(13.75,59.75){\line(1,0){14.5}}
\multiput(28.25,59.75)(.0336990596,.0415360502){319}{\line(0,1){.0415360502}}
\multiput(28.5,60)(.0336879433,-.0425531915){282}{\line(0,-1){.0425531915}}
\put(4.25,76){$v_2$}
\put(2.5,43){$v_3$}
\put(12.,52){$v_1$}
\put(25,52){$v_4$}
\put(40,76){$v_3$}
\put(40,42){$v_6$}
\put(60,83.25){\line(0,-1){65}}
\put(51.25,77.5){\line(1,0){70}}
\put(63,81){$v_1$}
\put(73,81){$v_2$}
\put(83,81){$v_3$}
\put(93,81){$v_4$}
\put(103,81){$v_5$}
\put(113,81){$v_6$}
\put(52,72){$v_1$}
\put(52,62){$v_2$}
\put(52,52){$v_3$}
\put(52,42){$v_4$}
\put(52,32){$v_5$}
\put(52,22){$v_6$}
\put(63,71){$0$}
\put(73,71){$1$}
\put(83,71){$1$}
\put(93,71){$1$}
\put(103,71){$0$}
\put(113,71){$0$}
\put(62,61){$-1$}
\put(73,61){$0$}
\put(83,61){$0$}
\put(93,61){$0$}
\put(103,61){$0$}
\put(113,61){$0$}
\put(62,51){$-1$}
\put(73,51){$0$}
\put(83,51){$0$}
\put(93,51){$0$}
\put(103,51){$0$}
\put(113,51){$0$}
\put(62,41){$-1$}
\put(73,41){$0$}
\put(83,41){$0$}
\put(93,41){$0$}
\put(101,41){$-1$}
\put(111,41){$-1$}
\put(63,31){$0$}
\put(73,31){$0$}
\put(83,31){$0$}
\put(93,31){$1$}
\put(103,31){$0$}
\put(113,31){$0$}
\put(63,21){$0$}
\put(73,21){$0$}
\put(83,21){$0$}
\put(93,21){$1$}
\put(103,21){$0$}
\put(113,21){$0$}
\put(53.75,10){Figure 4}%
\end{picture}

We can now restate the Proposition 4 as follows involving uniform\break d.a.matrix.

\begin{prop} For a graph $G$ the following statements are equivalent.\\
(i) $G$ is bipartite.\\
(ii)$G$ has a uniform d.a. matrix.\end{prop}
\section{(1,2)-Uniform Directional adjacency Matrix of a Graph}

Let $V=\{v_1,v_2,\ldots,v_p\}$ be the vertex set of a graph $G=(V,E)$, and $M=[a_{ij}]$ be a directional adjacency matrix of $G$.\\

If all nonzero entries in the $i^{th}$ row of $M$ are $1$ or $-1$,
we indicate this by $r_{M}(i)=1$ or $-1$ accordingly. It may
happen that in an $i^{th}$ row, the non-zero entries may contain
both $1$ and $-1$. In this case we denote this fact by
$r_M(i)=\{1,-1\}$. A d.a.matrix $M=[a_{ij}]$ of $G$ is called
$(1,2)$-\emph{uniform} if for any two adjacent vertices $v_i$ and
$v_j$, one of the following holds:\\
(i) $r_M(i)=1$, $r_M(j)=-1$\\
(ii) $r_M(i) =1$, $r_M(j)=\{1,-1\}$\\
(iii)$r_M(i) =-1$, $r_M(j)=\{1,-1\}$.\\\\

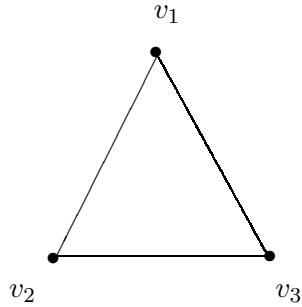
\unitlength 1mm 
\linethickness{0.4pt}
\ifx\plotpoint\undefined\newsavebox{\plotpoint}\fi 
\begin{picture}(107.75,53.25)(0,0)
\put(27.25,47.5){\circle*{1.5}}
\put(13.75,20.25){\circle*{1.5}}
\put(42.25,20.5){\circle*{1.5}}
\put(27.5,47){\line(-1,-2){13.75}}
\multiput(27.25,47.5)(.0336757991,-.0616438356){438}{\line(0,-1){.0616438356}}
\put(42,20.5){\line(-1,0){28.5}}
\put(27,52){$v_1$}
\put(8,15){$v_2$}
\put(43,15){$v_3$}
\put(65.5,52){\line(0,-1){38}}
\put(56.25,45.5){\line(1,0){51.5}}
\put(71,49){$v_1$}
\put(84,49){$v_2$}
\put(99,49){$v_3$}
\put(57,40){$v_1$}
\put(57,28.75){$v_2$}
\put(57,18){$v_3$}
\put(71,38){$0$}
\put(84,38){$1$}
\put(99,38){$1$}
\put(68,28){$-1$}
\put(84,28){$0$}
\put(96,28){$-1$}
\put(68,18){$-1$}
\put(84,18){$1$}
\put(99,18){$0$}
\end{picture}

\unitlength 0.8mm 
\linethickness{0.4pt}
\ifx\plotpoint\undefined\newsavebox{\plotpoint}\fi 
\hskip .25cm
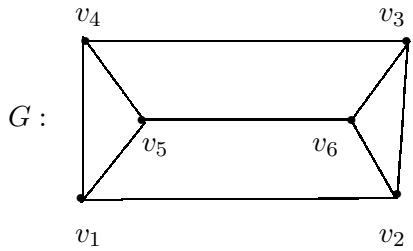
\begin{picture}(148.75,107.25)(0,0)
\put(6.75,100.25){\circle*{1.5}} \put(50.5,87.25){\circle*{1.5}}
\put(7,100.25){\line(1,0){53}}
\multiput(60,100.25)(-.0336879433,-.0460992908){282}{\line(0,-1){.0460992908}}
\put(50.5,87.25){\line(-1,0){34.25}}
\multiput(16.25,87.25)(-.0336700337,.0462962963){297}{\line(0,1){.0462962963}}
\put(6.25,101){\line(0,-1){27}}
\multiput(6.25,74)(6.46875,.03125){8}{\line(1,0){6.46875}}
\multiput(58,74.25)(.03365385,.49519231){52}{\line(0,1){.49519231}}
\multiput(50.5,87)(.03372093,-.059302326){215}{\line(0,-1){.059302326}}
\multiput(16.5,86.75)(-.0336538462,-.0424679487){312}{\line(0,-1){.0424679487}}
\put(59.5,100.25){\circle*{1.5}}
\put(58,75){\circle*{1.5}}
\put(6,74.25){\circle*{1.5}}
\put(16,87.25){\circle*{1.5}}
\put(5,67){$v_1$}
\put(55,67){$v_2$}
\put(55,103.5){$v_3$}
\put(5,103.5){$v_4$}
\put(16,82){$v_5$}
\put(44,82){$v_6$}
\put(-6,86.25){$G:$}
\put(77.5,99){\line(0,-1){53.25}}
\put(77.5,107.25){\line(0,-1){77}}
\put(69.5,101.25){\line(1,0){79.25}}
\put(82,106){$v_1$}
\put(94,106){$v_2$}
\put(106,106){$v_3$}
\put(118,106){$v_4$}
\put(130,106){$v_5$}
\put(142,106){$v_6$}
\put(69,93){$v_1$}
\put(69,81){$v_2$}
\put(69,69){$v_3$}
\put(69,57){$v_4$}
\put(69,45){$v_5$}
\put(69,33){$v_6$}
\put(77.5,46.75){\line(0,-1){3.5}}
\put(77.5,43.25){\line(0,-1){3.25}}
\put(82,93){$0$}
\put(94,93){$1$}
\put(106,93){$0$}
\put(118,93){$1$}
\put(130,93){$1$}
\put(142,93){$0$}
\put(80,81){$-1$}
\put(94,81){$0$}
\put(104,81){$-1$}
\put(118,81){$0$}
\put(130,81){$0$}
\put(140,81){$-1$}
\put(82,69){$0$}
\put(94,69){$1$}
\put(106,69){$0$}
\put(118,69){$1$}
\put(130,69){$0$}
\put(142,69){$1$}
\put(80,57){$-1$}
\put(94,57){$0$}
\put(104,57){$-1$}
\put(118,57){$0$}
\put(130,57){$1$}
\put(142,57){$0$}
\put(80,45){$-1$}
\put(94,45){$0$}
\put(106,45){$0$}
\put(116,45){$-1$}
\put(130,45){$0$}
\put(140,45){$-1$}
\put(82,33){$0$}
\put(94,33){$1$}
\put(104,33){$-1$}
\put(118,33){$0$}
\put(130,33){$1$}
\put(142,33){$0$}
\put(70,20){Figure 5}%
\end{picture}\vspace{-1.5pc}

For example, the d.a.matrix of $K_3$ and the graph $G$ in Figure 5
are $(1,2)$-uniform.\\\\
\emph{Note}:If $M$ is a $(1,2)$-uniform directional adjacency matrix of a graph $G$, then for any two adjacent vertices $v_i$ and $v_j$, $r_M(i) \neq r_M(j)$. One can restate Proposition 5 as follows involving $(1,2)$-uniform d.a.matrix.

\begin{prop} For a graph $G$ the following statements are equivalent.\\
(i) $G$ is 3-colorable.\\
(ii)$G$ has a $(1,2)$-uniform d.a.matrix.\end{prop}

\end{document}